\documentclass{article}
\usepackage{paralist} 

\usepackage{amssymb} 
\usepackage{amsmath,tikz,wasysym,hyperref} 
\usepackage[labelformat=simple]{subcaption}
\usepackage{soul}
\usepackage{cite}
\usepackage{cancel}
\tikzset{
every node/.style={draw, circle, inner sep=2pt}
}

\usepackage{amsthm}

\newtheorem{theorem}{Theorem}[section]
\newtheorem{proposition}[theorem]{Proposition}
\newtheorem{corollary}[theorem]{Corollary}
\newtheorem{lemma}[theorem]{Lemma}

\theoremstyle{definition}
\newtheorem{example}[theorem]{Example}

\newtheorem{remark}[theorem]{Remark}
\newtheorem{definition}[theorem]{Definition}

\newcommand{\hide}[1]{}

\newcommand{\npmatrix}[1]{\left( \begin{matrix} #1 \end{matrix} \right)}
\newcommand{\R}{\mathbb{R}}

\newcommand{\mr}{\mathrm{mr}}

\newcommand{\tr}{\operatorname{tr}}

\newcommand{\mc}[1]{\mathcal{#1}}
\newcommand{\trans}{^\top}
\newcommand{\dunion}{\oplus}
\newcommand{\calS}{\mathcal{S}}
\newcommand{\spec}{\operatorname{spec}}

\newcommand{\mptncl}{\mathcal{S}^{\rm cl}}
\newcommand{\SnR}{S_n(\mathbb{R})}
\newcommand{\mei}{\mathcal{E}}
\newcommand{\rank}{\operatorname{rank}}

\newcommand{\bzero}{{\bf 0}}
\newcommand{\bone}{{\bf 1}}
\newcommand{\bb}{{\bf b}}
\newcommand{\bu}{{\bf u}}

\newcommand{\by}{{\bf y}}
\newcommand{\bx}{{\bf x}}

    \definecolor{helena}{rgb}{.2,.8,.4}
    \definecolor{polona}{rgb}{.4,.2,.8}
    \definecolor{jephian}{rgb}{.3,.4,.1}
   \definecolor{todo}{rgb}{.8,.2,.2}

\title{On the inverse eigenvalue problem for block graphs}
\author{Jephian C.-H.~Lin
\thanks{Department of Applied Mathematics, National Sun Yat-sen University, Kaohsiung 80424, Taiwan (jephianlin@gmail.com)}
\and 
Polona Oblak
\thanks{Faculty of Computer and Information Science, University of Ljubljana, Ve\v cna pot 113, SI-1000 Ljubljana, Slovenia (polona.oblak@fri.uni-lj.si)}
\and 
Helena \v{S}migoc
\thanks{School of Mathematics and Statistics, University College Dublin, Belfield, Dublin 4, Ireland (helena.smigoc@ucd.ie)}
}

\begin{document}

\setcounter{footnote}{1}
\maketitle 

\begin{abstract}
The inverse eigenvalue problem of a graph $G$ aims to find all possible spectra for matrices whose $(i,j)$-entry, for $i\neq j$, is nonzero precisely when $i$ is adjacent to $j$.  In this work, the inverse eigenvalue problem is completely solved for a subfamily of clique-path graphs, in particular for lollipop graphs and generalized barbell graphs. For a matrix $A$ with associated graph $G$, a new technique utilizing the strong spectral property is introduced, allowing us to construct a matrix $A'$ whose graph is obtained from $G$ by appending a clique while arbitrary list of eigenvalues is added to the spectrum.  Consequently, many spectra are shown realizable for block graphs.  
\end{abstract}


\noindent{\bf Keywords:} 
Symmetric matrix; Inverse Eigenvalue Problem; Strong spectral property; Graph; Block graph
\medskip

\noindent{\bf AMS subject classifications:}
05C50, 
15A18, 
15B57, 
65F18. 

\section{Introduction}

Let $G=(V, E)$ be a simple graph on $n$ vertices. Define $\calS(G)$ as the set of all $n \times n$ real symmetric  matrices $A=(a_{ij})$ such that for $i \ne j$, $a_{ij}\ne 0$ if and only if $\{i,j\} \in E$.   Note that there is no restriction on the diagonal entries of matrix $A$.  The inverse eigenvalue problem of a graph $G$ (IEP-$G$) asks what possible spectra occur among matrices in $\calS(G)$.  The IEP-$G$ is motivated from the theory of vibrations \cite{MR2102477,MR1908601} and can be viewed as a discrete version of the question:  What kind of vibration behaviors (spectra) are allowed on a given structure (graph)?  On the other hand, the IEP-$G$ is a fundamental question in matrix theory and studies the possible spectra of a matrix with a given zero-nonzero pattern.

The IEP-$G$ is solved for only a handful of families of graphs. In particular, the IEP-$G$ for paths has been solved and many additional properties of tridiagonal matrices were studied \cite{MR382314, MR447279, MR447294}.  The solutions of the IEP-$G$ for generalized star graphs \cite{MR2022294} and cycles \cite{MR583498} are also known.  The IEP-$G$ for complete graphs and small graphs (up to $4$ vertices) was solved by explicit matrix construction in \cite{MR3118942,MR3291662}.  Recently, Barrett et al.\ introduced new techniques to the problem based on  the strong spectral property (SSP), and solved the IEP-$G$ for graphs with up to $5$ vertices \cite{MR4074182}. 

This paper introduces two techniques for constructing matrices of a given graph.  Section~\ref{sec:dup} considers the operation of duplicating a vertex into a clique with the same neighborhood; see, e.g., Figure~\ref{k-duplication} or \ref{fig:117150}. For a graph $G$ of order $n$ and a graph $H$ on $m$ vertices obtained from $G$ by a series of such duplications,   Corollary~\ref{thm:blowup} shows that any 
list of $m$ real numbers having at least $n$ distinct elements is realizable by some matrix in $\calS(H)$. As a consequence, the IEP-$G$ for lollipops and barbell graphs is solved.

Section~\ref{sec:ssp} considers a generalization of the SSP and establishes a clique-appending lemma.  Suppose a graph $H$ is obtained from a graph $G$ by appending a leaf.  It is known that if $\Lambda$ is a spectrum realizable by a matrix with the SSP in $\calS(G)$, then $\Lambda\cup\{\lambda\}$ is realizable by a matrix with the SSP in $\calS(H)$ for any $\lambda\notin\Lambda$ \cite{MR4074182}.  We generalize this behavior to the case where $H$ is obtained from $G$ by appending a clique $K_k$ while at the same time an eigenvalue of multiplicity $k$ is added to the spectrum of the corresponding matrix in $\calS(G)$.


Utilizing the tools developed in Sections~\ref{sec:dup} and \ref{sec:ssp}, we provide partial solution to the IEP-$G$ for block graphs in Section~\ref{sec:block}.

\subsection{Preliminaries}

Below we define the main  notation used in the paper, that is predominantly standard. 

By $S_n(\R)$ we will denote the set of all $n\times n$ symmetric matrices.  
For a simple graph $G=(V(G),E(G))$, $|G|=|V(G)|$ will denote the order of $G$, and recall that $\calS(G)$ denotes the set of all $A=(a_{ij}) \in S_{|G|}(\R)$ such that for $i \ne j$, $a_{ij}\ne 0$ if and only if $\{i,j\} \in E(G)$. Moreover, $\mptncl(G)$ will denote the  topological closure of $\calS(G)$, i.e., the set of $A=(a_{ij}) \in S_{|G|}(\R)$ with $(i,j)$-entry nonzero only when $i=j$ or $\{i,j\}\in E(G)$.  Note that $\mptncl(G)$ is a linear subspace in $\SnR$ of dimension $n+|E(G)|$.

For a matrix $A$, $A(i)$ will denote the submatrix of $A$ with the $i$th row and the $i$th column removed, and $A\oplus B$ will denote the direct sum of matrices $A$ and $B$.
The $n \times n$ identity matrix will be denoted by $I_n$, and the $m \times n$ zero matrix by $O_{m,n}$. In both cases the indices will be omitted if they are clear from the context. 

Suppose that  $A\in \SnR$ has distinct eigenvalues $\lambda_1, \ldots,\lambda_{q}$ with multiplicities $m_1,\ldots, m_q$, respectively.  The spectrum of $A$ will be denoted by $\spec(A)=\{\lambda_1^{(m_1)},\ldots,\lambda_q^{(m_q)}\}$, where $\lambda^{(k)}$ denotes $k$ copies of $\lambda$. The \emph{multiplicity list} of $A$ is defined to be a list of multiplicities $\{m_1,\ldots,m_q\}$, in no particular order. We say that the multiplicity list $\{m_1,\ldots,m_q\}$ is \emph{spectrally arbitrary} in $\calS (G)$ if for any $\lambda_1, \ldots,\lambda_{q}$, the spectrum $\{\lambda_1^{(m_1)},\ldots,\lambda_q^{(m_q)}\}$ is realizable by a matrix in $\calS (G)$. Note that this definition of spectral arbitrariness is stronger than the one in some other works (see e.g., \cite{2017arxiv170802438,MR4074182}), where it is assumed that the multiplicity lists are ordered.

A matrix $A\in S_n(\R)$ has the \emph{strong spectral property} (\emph{SSP}) if the zero matrix $X=O$ is the only symmetric matrix $X$ satisfying $A\circ X=I \circ X=O$ and $[A,X]=O$, where $\circ$ denotes the entry-wise product of matrices and $[A,X]=AX-XA$. The strong spectral property of a matrix was first defined in \cite{MR3665573} and has been proven useful in expanding any information on the IEP-$G$ for  a given graph to information for its supergraph.  

\begin{theorem}
\label{thm:sspsuper}
{\rm \cite
{MR3665573}}
Let $H$ be a graph and $G$ a spanning subgraph of $H$.  Suppose $A\in\calS(G)$ is a matrix with spectrum $\Lambda$ and the SSP.  Then for any $\epsilon>0$ there is a matrix $A'\in\calS(H)$ with the SSP such that $\spec(A)=\spec(A')$ and $\|A-A'\|<\epsilon$.
\end{theorem}

We follow standard notation for basic graphs encountered in this work, i.e., $K_n$ denotes the complete graph on $n$ vertices, 
$P_n$ denotes the path on $n$ vertices and $S_n$ the star on $n$ vertices.
 
 If $v\in V(G)$,  let $G-v$ denote the subgraph of $G$  obtained from $G$ by removing the vertex $v$ and all edges incident to it.  A connected graph $G$, $|G| \geq 2$, is called \emph{2-connected} if $G-v$ is connected for any $v\in V(G)$. 
Let $G$ and $H$ be two graphs. We denote the disjoint union of $G$ and $H$ by $G \oplus H$. If each of $G$ and $H$ has a vertex labeled as $v$, then the \emph{vertex-sum} $G\oplus_v H$ of $G$ and $H$ at  $v$ is the graph obtained from $G\dunion H$ by identifying the two vertices labeled by $v$.

\section{Vertex duplication}
\label{sec:dup}
In this section we will develop a method that will allow us to replace a vertex $v \in V(G)$ in a graph $G$ with $k$ mutually adjacent vertices whose neighborhood  in $G-v$ is the same as that of $v$, while preserving some control on the eigenvalues of $A \in \calS(G)$. Note that the resulting graph $H$ has $|G|+k-1$ vertices. We will call this operation \emph{$k$-duplication} of $v$ in $G$. 
In particular, for  $A=(a_{ij}) \in \calS(G)$, we will  apply the following two lemmas to create a matrix $C\in\calS(H)$.

\begin{lemma}
\label{HS04}
{\rm \cite{MR2098598}}
Let $B$ be an $k \times k$ symmetric matrix with eigenvalues $\mu_1,\ldots,\mu_k$ and let $\bu$ be an eigenvector corresponding to $\mu_1$, normalized so that $\bu\trans\bu=1$.  Let $A$ be an $n \times n$ symmetric matrix with a diagonal element $\mu_1$
\[
  A=\npmatrix{A_1 & \bb \\
               \bb\trans & \mu_1}
\]
 and eigenvalues $\lambda_1, \ldots, \lambda_n$.
 Then the matrix 
\[
  C=\npmatrix{A_1 & \bb\bu\trans \\
               \bu\bb\trans & B}\]
               has eigenvalues $\lambda_1,\ldots,\lambda_n,\mu_2,\ldots, \mu_k$.
\end{lemma}

Choosing $B \in \calS (K_k)$ in Lemma \ref{HS04} results in $k$-duplication in the associated graph.  While this lemma can be applied more generally, we will take particular advantage of the fact that the IEP-$G$ for complete graphs is solved (see e.g.,~\cite{MR3118942}). Furthermore, we will need information on possible patterns of the eigenvectors of matrices in $\calS (K_k)$, as outlined in the following Lemma.

\begin{lemma}
\label{complete}
{\rm \cite{MR3208410}}
For any given list of real numbers $\sigma=\{\mu_1,\mu_2,\ldots,\mu_k\}$, $\mu_1 \neq \mu_2$, there exists $B \in \calS(K_k)$ with spectrum $\sigma$.

Furthermore, given any zero-nonzero pattern of a vector in $\R^k$ that contains at least two nonzero elements, $B$ can be chosen so that it has an eigenvector corresponding to $\mu_1$ with that given pattern.
\end{lemma}

Thus, inserting $B \in \calS(K_k)$ with an eigenvalue $\mu_1$ in Lemma \ref{HS04} will show that any spectrum of the form 
\[\spec(A) \cup \{\mu_2,\ldots,\mu_k \}\]
is realizable in $\calS(H)$. Special attention needs to be paid to the case when 
$a_{vv} = \mu_1 = \mu_2 = \cdots = \mu_k$.
In this case, we necessarily have $B = a_{vv}I_k$, which is not in $\calS(K_k)$ whenever $k\geq 2$.  
Typically, this situation can be avoided by replacing $A$ with a matrix of the same pattern and eigenvalues, but different diagonal elements. In particular, Lemma~\ref{lem:diag} states that if we require $A$ to have distinct eigenvalues, then we can avoid any prescribed finite set of real numbers on the diagonal.

\begin{lemma}\label{lem:diag}
Let $\sigma=\{\lambda_1,\ldots,\lambda_n\}$ be a set of distinct real numbers and $\mc{F}$  a finite set of real numbers. For any connected graph $G$ on $n\geq 2$ vertices there exists $A \in\calS(G)$ with the SSP such that $\spec(A)=\sigma$, and none of the diagonal entries of $A$ is contained in $\mc{F}$. 
\end{lemma}
\begin{proof}
First we claim that the lemma is true for stars  with at least two vertices.  Let $G$ be a star on $n$ vertices with vertex $1$ as its center.  Let $\mu_1,\ldots,\mu_{n-1}$ be real numbers such that 
\[\lambda_1<\mu_1<\lambda_2<\cdots<\lambda_{n-1}<\mu_{n-1}<\lambda_n,\]
$\mu_i\notin\mc{F}$ for $i=1,\ldots,n-1$, and $\sum_{i=1}^n \lambda_i - \sum_{i=1}^{n-1}\mu_i \notin \mc{F}$.  Since $\mc{F}$ is assumed to be finite, such numbers exist.  By \cite[Lemmas~2.1 and 2.2]{MR3034535}, there is a matrix $A = (a_{ij})\in\calS(G)$ with the SSP such that $\spec(A) = \{\lambda_1,\ldots,\lambda_n\}$ and $\spec(A(1)) = \{\mu_1,\ldots,\mu_{n-1}\}$.  Notice that $\spec(A(1))$ is also equal to $\{a_{22},\ldots, a_{nn}\}$, so $a_{ii}\notin\mc{F}$ for $i=2,\ldots,n$, by construction.  Also, 
\[a_{11} = \tr(A) - \tr(A(1)) = \sum_{i=1}^n \lambda_i - \sum_{i=1}^{n-1}\mu_i \notin \mc{F}.\] 

As an intermediate step, we claim that every connected graph $G$ on $n\geq 2$ vertices has a spanning subgraph $H$ such that each component of $H$ is a star with at least two vertices.  This can be seen by induction on $n$.  For $n=2$, the claim is obviously true for the only connected graph $K_2$.  Now suppose the statement is true for all graphs satisfying $2 \leq |V(G)| \leq n-1$.  Let $G$ be a connected graph on $n$ vertices, and let $v \in V(G)$  be such vertex that $G-v$ remains connected.  By the induction hypothesis, there is a spanning subgraph $H'$ of $G-v$ whose components are stars with at least two vertices.  Pick a neighbor of $v$ in $G$, say $w$, and let $S$ be the connected component of $H'$ that contains $w$.  If $S$ along with the edge $\{v,w\}$ is a star, then let $H$ be obtained from $H'$ by adding  the vertex $v$ and the edge $\{v,w\}$.  If $S$ along with the edge $\{v,w\}$ is not a star, then $S-w$ is still a star with at least two vertices. Thus, let $H$ be obtained from $H'$ by removing $w$ and adding the $K_2$ induced on the vertices $v$ and $w$.  In either case, $H$ is the desired spanning subgraph of $G$.

To complete the proof, let $G$ be a connected graph on $n\geq 2$ vertices, $\sigma=\{\lambda_1,\ldots,\lambda_n\}$ a set of distinct real numbers, and $H$ a spanning subgraph of $G$ whose components are stars with at least two vertices.  We may write $H$ as a disjoint union of stars $S_1,\ldots,S_\ell$ of orders $k_1,\ldots,k_\ell$, respectively.  Partition $\sigma$ into $\ell$ parts $\sigma_1,\ldots,\sigma_\ell$ of orders $k_1,\ldots,k_\ell$, respectively.  Thus, we have already proved above that for each $i=1,\ldots,\ell$, there exists a matrix $A_i\in\calS(S_i)$ with the SSP such that $\spec(A_i) = \sigma_i$ and the diagonal entries of $A_i$ avoid $\mathcal{F}$.  Let $A=\bigoplus_{i=1}^\ell A_i$ be the direct sum of $A_i$'s.  Clearly, $\spec(A) = \sigma$ and the diagonal entries of $A$ avoid $\mathcal{F}$.  Also, by \cite[Theorem~34]{MR3665573} $A$ has the SSP since the spectra of $A_i$'s are mutually disjoint.  By Theorem~\ref{thm:sspsuper}, for any $\epsilon>0$, there is a matrix $A'\in\calS(G)$ with the SSP such that $\spec(A')=\sigma$ and $\|A-A'\|\leq \epsilon$.  When $\epsilon$ is chosen small enough, the diagonal entries of $A'$ remain disjoint from $\mathcal{F}$.
\end{proof}

Let $G$ be a graph on vertices $V(G)=\{v_1,\ldots,v_n\}$.  Define the \emph{(closed) blowup} of $G$ with respect to $n$ positive integers $m_1,\ldots,m_n$ by a graph obtained from $G$ by $m_i$-duplication of $v_i$ for $i=1,\ldots,n$ sequentially.

\begin{lemma}\label{lem:blowup}
Let $G$ be a graph with $|V(G)|=n$ and $H$ a blowup of $G$ with $|V(H)|=m$.  If $\sigma'$ is a multiset of $m-n$ real numbers and $A\in\calS(G)$ is a matrix whose diagonal entries avoid elements in $\sigma'$, then there is a matrix $A'\in\calS(H)$ with spectrum $\spec(A)\cup\sigma'$.  
\end{lemma}
\begin{proof}
Let $V(G)=\{v_1,\ldots,v_n\}$ and let $H$ be a blowup of $G$ obtained by $m_i$-duplication of $v_i$ for $i=1,\ldots,n$, $\sum_{i=1}^n m_i=m$. Partition $\sigma'$ into $n$ parts $\sigma'_1,\ldots,\sigma'_n$ of orders $m_1-1,\ldots,m_n-1$, respectively.  For $i=1,\ldots,n$ choose a matrix $B_i\in \calS (K_{m_i})$ with spectrum $\sigma'_i \cup \{a_{ii}\}$ and a nowhere-zero eigenvector $\bu_i\in\R^{m_i}$ corresponding to the eigenvalue $a_{ii}$. Such matrices exist by Lemma \ref{complete}. 
 Applying Lemma \ref{HS04} on each vertex we obtain a matrix $A'\in\calS (H)$ with $\spec(A')=\spec(A)\cup\sigma'$.
\end{proof}

\begin{theorem}\label{thm:blowup}
Let $G$ be a connected graph on $n\geq 2$ vertices and $H$ a blowup of $G$ with $|V(H)| = m$.  Suppose $\sigma$ is a multiset with $n$ distinct real numbers, and $\sigma'$ is any multiset with $m-n$ real numbers.  Then $\sigma\cup\sigma'$ is the spectrum of  some matrix in $\calS(H)$.
\end{theorem} 
\begin{proof}
This follows from and Lemmas~\ref{lem:diag} and \ref{lem:blowup}.
\end{proof}

\hide{\begin{proof}
Let $V(G)=\{v_1,\ldots,v_n\}$ and let $H$ be a blowup of $G$ obtained by $m_i$-duplication of $v_i$ for $i=1,\ldots,n$, $\sum_{i=1}^n m_i=m$. Partition $\sigma'$ into $n$ parts $\sigma'_1,\ldots,\sigma'_n$ of orders $m_1-1,\ldots,m_n-1$, respectively.
 Lemma \ref{lem:diag} guarantees the existence of a matrix $A=(a_{ij}) \in \calS (G)$ with the SSP and $\spec(A)=\sigma$ such that $a_{ii}\notin \sigma'$ for $i=1,\ldots,n$. Moreover, for $i=1,\ldots,n$ choose a matrix $B_i\in \calS (K_{m_i})$ with spectrum $\sigma'_i \cup \{a_{ii}\}$ and a nowhere-zero eigenvector $\bu_i\in\R^{m_i}$ corresponding to the eigenvalue $a_{ii}$. Such matrices exist by Lemma \ref{complete}. 
 Applying Lemma \ref{HS04} on each vertex we obtain a matrix $C \in \calS (H)$ with $\spec(C)=\sigma\cup \sigma'$.
\end{proof}
}

The \emph{$(k,p)$-lollipop graph} $L_{k,p}$ is the graph on $k+p$ vertices obtained by adding an edge between a vertex in a complete graph $K_k$ and a leaf of a path graph $P_p$. See an example in Figure~\ref{lollipop}.  Since $L_{k,p}$ can also be viewed as a blowup of $P_{p+2}$ by $(k-1)$-duplication of one of its leaves, Corollary \ref{thm:blowup} resolves the IEP-$G$ for lollipop graphs.

 \begin{corollary}\label{cor:lollipop}
Let $\sigma$ be a multiset with $k+p$ elements, where $k\geq 2$.  Then $\sigma$ is a spectrum of a matrix in $\calS(L_{k,p})$ if and only if $\sigma$ contains at least $p+2$ distinct elements.
 \end{corollary}

\begin{proof}
Since $L_{k,p}$ has a unique shortest path on $p+2$ vertices, every matrix in $L_{k,p}$ has at least $p+2$ distinct eigenvalues by \cite[Theorem~3.2]{MR3118943}. Together with Theorem~\ref{thm:blowup} the result follows.
\end{proof}

\begin{center}
 \begin{figure}[htb]
 \centering
\begin{subfigure}{0.45\textwidth}
\centering
\begin{tikzpicture}[scale=0.7]
    \foreach \angle in {0, 60,120} {
              \draw  (\angle:1)--(\angle+180:1);
     }
    \foreach \angle in {0, 60,120,180,240,300} {
             \draw  (\angle:1)--(\angle+60:1);
              \draw  (\angle:1)--(\angle+120:1);}
   \foreach \angle in {60,120,180,240,300} {
    \node[fill=white] at (\angle:1) {};
     }
  	\node[fill=white] (1) at (1,0) {};
    \node[fill=white] (2) at (2,0) {};
  	\node[fill=white] (3) at (3,0) {};
  	\node[fill=white] (4) at (4,0) {};
    \draw  (1)--(2) -- (3)--(4);
   \end{tikzpicture}
  \caption{$L_{6,3}$}
  \label{lollipop}
  \end{subfigure}
\begin{subfigure}{0.45\textwidth}
\centering
\begin{tikzpicture}[scale=0.7]
  \foreach \angle in {0, 60,120} {
              \draw  (\angle:1)--(\angle+180:1);
     }
    \foreach \angle in {0, 60,120,180,240,300} {
             \draw  (\angle:1)--(\angle+60:1);
              \draw  (\angle:1)--(\angle+120:1);}
   \foreach \angle in {60,120,180,240,300} {
    \node[fill=white] at (\angle:1) {};
     }
  	\node[fill=white] (1) at (1,0) {};
    \node[fill=white] (2) at (2,0) {};
  	\node[fill=white] (3) at (3,0) {};
  	\node[fill=white] (4) at (4,0) {};
  	\node (5) at (4.866,0.5) {};
  	\node (6) at (4.855,-0.5) {};
     \draw  (1)--(2) -- (3)--(4) -- (5) -- (6) -- (4);
    \end{tikzpicture}
  \caption{${B_{6,2,3}}$}
  \label{fig:barbell}
  \end{subfigure}
     \caption{The $(6,3)$-lollipop graph $L_{6,3}$ is $5$-duplication of a path $P_5$ at its leaf. The $(6,2,3)$-barbell graph $B_{6,2,3}$ is obtained as $5$-duplication and $2$-duplication of leaves of $P_6$.}\label{k-duplication}
     \end{figure}
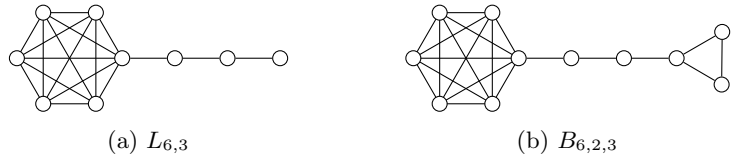
  \end{center}

Note that every matrix in $\calS(L_{k,p})$ has the the SSP \cite[Example~5.1]{MR4080669}. Together with Theorem~\ref{thm:sspsuper}  this leads to the next corollary.

\begin{corollary}
Let $G$ be a graph that contains a spanning subgraph isomorphic to $L_{k,p}$.  Then any spectrum with at least $p+2$ distinct elements is realizable by some matrix with the SSP in $\calS(G)$.
\end{corollary}

Similar arguments apply  for a generalized barbell graphs $B_{k',p,k''}=K_{k'} \oplus_v P_{p+2} \oplus_w K_{k''}$, where $v$ and $w$ are the leaves of $P_{p+2}$.
Note that $B_{k',p,k''}$  is a blowup of $P_{p+4}$ by $(k'-1)$-duplication and $(k''-1)$-duplication. See example on Figure~\ref{fig:barbell}.

\begin{corollary}\label{cor:barbell}
 Let  $\sigma$ be a multiset with $k'+p+{k''}$ elements, where $k' \geq 2$ and $k'' \geq 2$.  Then $\sigma$ is a spectrum of a matrix in $\calS(B_{k',p,k''})$ if and only if $\sigma$ contains at least $p+4$ distinct elements.
 \end{corollary}

\begin{example}
In \cite[Appendix~B]{2017arxiv170802438} the authors investigate spectral arbitrariness of graphs on at most six vertices, where some of the achieved multiplicities remain unsolved.

Firstly, Corollary \ref{cor:barbell} implies that for a generalized barbell graph $G_{130}=B_{3,0,3}$ any multiplicity list with at least 4 elements is spectrally arbitrary. In particular, ordered multiplicity lists $(1,1,3,1)$ and $(1,3,1,1)$ are spectrally arbitrary.

Moreover,  graphs $G_{117}$ and $G_{150}$ can be obtained as blowups of $K_{1,3}$ and $P_{4}$, respectively, see Figure \ref{fig:117150}. Theorem \ref{thm:blowup} implies that (unordered) multiplicity list $\{3,1,1,1\}$ is spectrally arbitrary for $G_{117}$ and $G_{150}$. Again, in particular, ordered multiplicity lists $(1,1,3,1)$ and $(1,3,1,1)$ are spectrally arbitrary for $G_{117}$ and $G_{150}$.

\begin{center}
 \begin{figure}[htb]
 \centering
\begin{subfigure}{0.45\textwidth}
\begin{tikzpicture}[scale=0.7]
\begin{scope}[shift={(-4,0)}]
    \node (0) at (0,0) {};
   \foreach \x in {90,234,306} 
     {  \draw  (0) -- (\x:1);
          }
        \node[fill=white] at (90:1) {};
        \node[fill=gray] at (234:1) {};
        \node[fill=gray] at (306:1) {};
    \end{scope}
     \node[rectangle, draw=none, right, align=left] at (-3,0) {$\longrightarrow$}; 
\node (0) at (0,0) {};
     \draw  (162:1) -- (234:1);
     \draw  (304:1) -- (18:1);
  	 \foreach \x in {90,162,234,306,18} 
     {  \draw  (0) -- (\x:1);
         \node[fill=white] at (\x:1) {};
          }
          \phantom{\node[fill=white] at (0,-2) {v};}
  \end{tikzpicture}
  \caption{$G_{117}$}
  \end{subfigure}
  \centering
\begin{subfigure}{0.45\textwidth}
\centering
\begin{tikzpicture}[scale=0.7]
   	\node[fill=gray] (1) at (-1,0) {};
    \node[fill=white] (2) at (-1,1) {};
  	\node[fill=gray] (3) at (-1,2) {};
  	\node[fill=white] (4) at (-1,3) {};
    \draw  (1)--(2) -- (3)--(4);
   	\node[fill=white] (11) at (2.3,0) {};
    \node[fill=white] (12) at (3.7,0) {};
  	\node[fill=white] (13) at (3,1) {};
  	\node[fill=white] (14) at (2.3,2) {};
  	\node[fill=white] (15) at (3.7,2) {};
  	\node[fill=white] (16) at (3,3) {};
    \draw  (13)--(11)--(12) -- (13)--(14)--(15)--(16)--(14);
  \draw  (13)--(15);
  \node[rectangle, draw=none, right, align=left] at (0.5,1.5) {$\longrightarrow$}; 
  \end{tikzpicture}
  \caption{$G_{150}$}
  \end{subfigure}
     \caption{Blowup of $K_{1,3}$ obtained by 2-duplication of each gray vertex results in graph $G_{117}$. Similarly, $G_{150}$ is a blowup of $P_4$, obtained by two $2$-duplications of gray vertices. We name the graphs following An Atlas of Graphs \cite{MR1692656}.}\label{fig:117150}
     \end{figure}
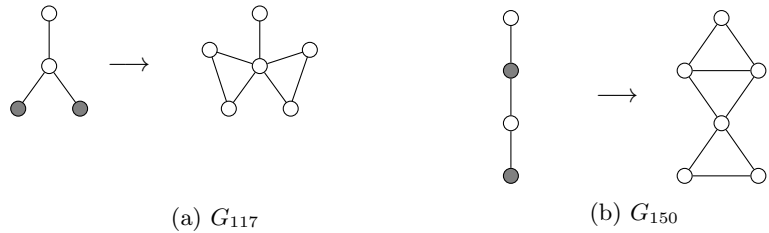
  \end{center}
\end{example}

\section{Strong spectral property with respect to a supergraph}
\label{sec:ssp}

In this section we extend the notion of the strong spectral property, that is built on the implicit function theorem for two transversally intersecting manifolds. Below we recall only a brief overview of the notions needed, and direct the reader to  \cite{MR3665573,MR2977424} for further details.

By definition, two manifolds intersect \emph{transversally} at a point $A$ if their normal spaces at that point have only trivial intersection. The implicit function theorem states that if two manifolds intersect at a point transversally, then this intersection will move continuously corresponding to any small perturbation of the two manifolds.  A version of the implicit function using the notion of a family of manifolds being smooth is given below.


\begin{theorem}
\label{thm:ift}
{\rm \cite[Theorem 3]{MR3665573}}
Let $\mathcal{M}_1(s)$ and $\mathcal{M}_2(t)$ be smooth families of manifolds in $\mathbb{R}^d$, where $s,t\in (-1,1)$, and $\mathcal{M}_1(0)$ and $\mathcal{M}_2(0)$ intersect transversally at $\by_0$. Then there exists a neighborhood $W\subseteq\mathbb{R}^2$ of the origin and a continuous function $f : W \rightarrow \mathbb{R}^d$ such that for each $\epsilon = (\epsilon_1, \epsilon_2) \in W$,
$\mathcal{M}_1(\epsilon_1)$ and $\mathcal{M}_2(\epsilon_2)$ intersect transversally at $f(\epsilon)$.
\end{theorem}

One family of manifolds that we are interested in is $\mptncl(G)$, where $G$ is a graph on $n$ vertices. Note that $\mptncl(G)$ is a subspace of $S_n(\R)$, so the tangent space to $\mptncl(G)$ at any of its points is equal to $\mptncl(G)$.  Then the normal space of $\mptncl(G)$ at $A\in\mptncl(G)$ is 
\[\mathcal{N}_{\mptncl(G),A} = \{X: X\in\mptncl(\overline{G}), I\circ X=O\},\]
and it does not depend on $A\in\mptncl(G)$.

The second family of manifolds that is relevant to our discussion is
\[\mei_\Lambda=\{B\in\SnR:\spec(B)=\Lambda\},\]
where $\Lambda$ is a multiset with $n$ real numbers.
If $A$ is an $n\times n$ matrix, then let $\mei_A=\mei_{\spec(A)}$.  By \cite[Lemma~7]{MR3665573}, the normal space of $\mei_A$ at $A$ is 
\[\mathcal{N}_{\mei_A,A}  = \{X\in\SnR: [A,X]=O\}.\]

From the discussion, a matrix $A$ in $\calS(G)$ has the SSP is equivalent to $\mei_A$ and $\mptncl(G)$ intersecting transversally at $A$.  Here we generalize this definition to a more flexible version of the SSP. 

\begin{definition}
Let $G$ be a graph and $A\in\calS(G)$.  Suppose $H$ is a supergraph of $G$ with $V(H)=V(G)$.  Then $A$ has the \emph{strong spectral property with respect to $H$} if $\mei_A$ and $\mptncl(H)$ intersect transversally at $A$.  Equivalently, $A$ has the SSP with respect to $H$ if the zero matrix $X=O$ is the only symmetric matrix that satisfies $X\in\mptncl(\overline{H})$, $I\circ X=O$, and $[A,X]=O$.
\end{definition}

By definition, a matrix $A\in\calS(G)$ has the SSP (in the classical sense) if and only if $A$ has the SSP with respect to $G$.  Also, if $H'$ is a supergraph of $H$ of the same order, then $A$ has the SSP with respect to $H$ implies $A$ has the SSP with respect to $H'$. 
 
\begin{example}
 For $n \geq 4$, let us observe the matrix 
 $$A=\npmatrix{0 & {\bf 1}\trans\\
 {\bf 1}& O_{n,n}}\in \calS(K_{1,n}),$$
 where ${\bf 1}\in \R^n$ denotes a vector with all entries equal to $1$.
That $A$ does not have the SSP, can be shown by choosing $X_1$ to be any $n \times n$ nonzero symmetric  matrix with zero diagonal and row sums equal to 0, for example, 
 $$
 X_1=\npmatrix{
 O_{2,2}&O_{2,n-2}&X_0\\
 O_{n-2,2}&O_{n-2,n-2}&O_{n-2,2}\\
 X_0&O_{2,n-2}&O_{2,2} 
 } \text{ where }X_0=\npmatrix{1 & -1 \\ -1 & 1},$$ 
letting $$X=\npmatrix{0&\bzero_n\trans\\\bzero_n&X_1}\in \mptncl(\overline{K_{1,n}}),$$ and observing that $I\circ X=O$, and $[A,X]=O$.
 
 Define $H$ to be a supergraph of $K_{1,n}$ obtained by removing the edges $\{2,3\}$, $\{3,4\}$, \ldots, $\{n,n+1\}$ from the complete graph $K_{n+1}$.
 Take any $Y \in \mptncl(\overline{H})$ with $I\circ Y=O$, and $[A,Y]=O$.
 Then $$Y=\npmatrix{0 & \bzero_n\trans\\
 \bzero_n& Y_1},$$ where $Y_1 \in \calS(P_n)$ with $I\circ Y_1 = O$ and $Y_1 \bone=\bzero$. This implies $Y_1=O $ and hence $A$ has the SSP with respect to $H$.
\end{example}
 
The generalized SSP leads to a generalized version of Theorem~\ref{thm:sspsuper}.

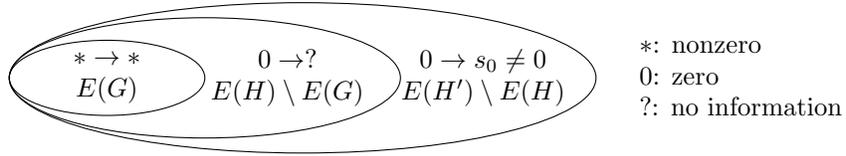
\begin{figure}[h]
\begin{center}
\begin{tikzpicture}
\node[rectangle, draw=none, align=center] at (0,0) {$*\rightarrow *$\\$E(G)$};
\node[rectangle, draw=none, align=center] at (2.4,0) {$0\rightarrow ?$\\$E(H)\setminus E(G)$};
\node[rectangle, draw=none, align=center] at (5.0,0) {$0\rightarrow s_0\neq 0$\\$E(H')\setminus E(H)$};
\draw (0,0) ellipse (1.3cm and 0.5cm);
\draw (1.3,0) ellipse (2.6cm and 0.8cm);
\draw (2.6,0) ellipse (3.9cm and 1cm);
\node[rectangle, draw=none, right, align=left] at (7,0) {
$*$: nonzero\\
$0$: zero\\
$?$: no information}; 
\end{tikzpicture}
\end{center}
\caption{Illustration of each entry of the matrix perturbed by Theorem~\ref{thm:perturb}.}
\label{fig:thmperturb}
\end{figure}

\begin{theorem}
\label{thm:perturb}
Let $G,H,H'$ be the graphs such that $V(G)=V(H)=V(H')$ 
and $E(G) \subseteq E(H) \subseteq E(H')$. Furthermore, let $A \in \calS(G)$ be a matrix that has the SSP with respect $H$. Then for any $\epsilon>0$, there is a matrix $A'\in\mptncl(H')$ such that $\spec(A')=\spec(A)$, $\|A-A'\|<\epsilon$, $A'$ has the SSP with respect to $H'$, and every entry of $A'$ that corresponds to an edge in $E(H')\setminus E(H)$ is nonzero.
\end{theorem}
\begin{proof}
Define a smooth family of manifolds 
\[\mathcal{M}(s)=\{B=\begin{pmatrix} b_{ij}\end{pmatrix}\in\mptncl(H'):b_{ij}=s \text{ if } (i,j)\in E(H')\setminus E(H)\}.\] 
By definition, $\mathcal{M}(0)=\mptncl(H)$.  Since $A$ has the SSP with respect to $H$, the two manifolds $\mei_A$ and $\mathcal{M}(0)$ intersect transversally at $A$.  According to Theorem~\ref{thm:ift}, $\mei_A$ and $\mathcal{M}(s)$ intersect transversally at $f(s)$ for any small $s$, and $f(s)$ is a continuous function of $s$ with $f(0)=A$. Since $\mathcal{M}(s)\subset\mptncl(H')$, the manifolds $\mei_A$  and $\mptncl(H')$ also intersect transversally at $f(s)$.

Thus, $A':=f(s_0)$ can be chosen arbitrarily close to $A$ with $s_0\neq 0$.  (See Figure~\ref{fig:thmperturb}.) Since $A'\in\mei_A$, $\spec(A')=\spec(A)$, and since $\mptncl(H')$ and $\mei_A$ intersect transversally at $A'$, $A'$ has the SSP with respect to $H'$.  Finally, the entries of $A'$ corresponding to edges in $E(H')\setminus E(H)$ are nonzero since $A'\in\mathcal{M}(s_0)$ and $s_0\neq 0$.
\end{proof}



The following proposition generalizes \cite[Theorem 34]{MR3665573}.
 
\begin{proposition}
\label{prop:dunion}
Let $A$ and $B$ be symmetric matrices with the SSP with respect to $H_A$ and $H_B$, respectively.  If $A$ and $B$ have no common eigenvalues, then $A\oplus B$ has the SSP with respect to $H_A\dunion H_B$.
\end{proposition}
\begin{proof}
Let 
\[X = \begin{bmatrix}X_{11}&X_{21}\trans\\ X_{21}&X_{22}\end{bmatrix}\]
be a matrix in $\mptncl(\overline{H_A\dunion H_B})$ such that $X\circ I=O$ and $[X, A\oplus B]=O$.  Equivalently, we have 
\[\begin{array}{cc}
X_{11}\circ I = O, & X_{22}\circ I = O, \\{}
[X_{11}, A] = O, & [X_{22}, B] = O, \\
\end{array}\]
and $BX_{21} = X_{21}A $.  Since $A$ and $B$ have no common eigenvalues, the last equality implies $X_{21} = O$ (see e.g.~\cite[Theorem~4.4.6]{MR1288752}). Since $A$ and $B$ have the SSP with respect to $H_A$ and $H_B$, respectively, both $X_{11}$ and $X_{22}$ have to be zero matrices.  Consequently, $X=O$ and $A\oplus B$ has the SSP with respect to $H_A\dunion H_B$.
\end{proof}

 Extending a graph by adding a vertex has been considered for example in \cite[Theorem~36]{MR3665573}, \cite[Theorem~6.13]{MR4074182}. Here we extend those results by adding a clique instead of a vertex to a graph, while adding an eigenvalue of multiplicity higher than $1$ to the spectrum.



\begin{theorem}\label{thm:addeig}
Let $G$ be a graph, $A\in\calS(G)$ with the SSP, $v \in V(G)$, and $\lambda\in\mathbb{R}$ with $\lambda\not\in \spec(A)\cup \spec(A(v))$.  Then for any given positive integer $s$, there exists a matrix $A'\in\calS(G\oplus_v K_{s+1})$ having the SSP such that $$\spec(A')=\spec(A)\cup\{\lambda^{(s)}\}.$$  Moreover, $A'$ can be chosen to be arbitrarily close to $A\oplus\lambda I_s$ while the $(i,i)$-entry of $A'$ for any $i\in V(K_{s+1})\setminus\{v\}$ is different from $\lambda$.
\end{theorem}

\begin{proof}
Choose a matrix $A=(a_{ij}) \in \calS(G)$ with the SSP, $v \in V(G)$, and $\lambda \not\in \spec(A)\cup\spec(A(v))$, as assumed in the theorem. The matrix $\lambda I_s$ has the SSP with respect to $K_s$, and $A$ and $\lambda I_s$ have no common eigenvalues. By Proposition~\ref{prop:dunion} the matrix $\hat{A}=A\oplus\lambda I_{s} \in \calS(G \dunion s K_1)$ has the SSP with respect to $H=G\dunion K_s$.

Let $\alpha$ be the set of all edges of the form $\{v,w\}$, where $w \in V(K_{s})$, and let $H'$ be a supergraph of $H$, obtained from $H$ by adding all the edges from $\alpha$, i.e., $H'=G\oplus_v K_{s+1}$.  By Theorem~\ref{thm:perturb}, for any $\epsilon>0$, there exists a matrix $A'=(a'_{ij})\in\mptncl(G\oplus_v K_{s+1})$ with the SSP with respect to $H'$, such that
 $\spec(A') =\spec(\hat A)= \spec(A)\dunion\{\lambda^ {(s)}\}$, $\|\hat A - A'\|<\epsilon$,
and $a'_{ij} \ne 0$ for all $\{i,j\}\in\alpha$.
Since the perturbation $\hat A - A'$ can be made to be arbitrarily small, all nonzero entries of $\hat{A}$ remain nonzero in $A'$. 


Finally, we show that the $(i,j)$-entry of $A'-\lambda I$ is nonzero whenever $i,j\in V(K_s)$.  This argument will allow us to conclude that $a'_{ij}\neq 0$ (if $i\neq j$) and $a_{ij}\neq \lambda$ (if $i=j$) for any $i,j\in V(K_s)$.  Since $A(v)-\lambda I$ is nonsingular and the perturbation is small, $A'$ can be chosen so that $A'[W]-\lambda I$ is nonsingular for $W=V(G)\setminus\{v\}$.  Pick any $i,j\in V(K_s)$, and consider the $(|G|+1) \times (|G|+1)$ submatrix $B$ of $A'-\lambda I$ induced on rows $V(G)\cup\{i\}$ and columns $V(G)\cup\{j\}$. Such matrix $B$ has the form 
\[B=\begin{bmatrix}
A'[W] - \lambda I & {\bf a} & \bzero \\
{\bf a}\trans & a'_{vv}-\lambda & b \\
\bzero\trans &c & a'_{ij} - \delta_{ij}\lambda
\end{bmatrix},\]
where $b,c \in \R$ are nonzero,  ${\bf a}\in \R^{|G|-1}$ and  $\delta_{ij}$ stands for the Kronecker delta function, i.e. $\delta_{ij}=0$ if $i\neq j$ and $\delta_{ij}=1$ if $i=j$.  Since $\lambda$ is an eigenvalue of $A'$ of multiplicity $s$, it follows that $\rank(B) \leq \rank(A'-\lambda I)=|G|$.
If $a'_{ij} - \delta_{ij}\lambda = 0$, then 
\[|G| \geq \rank(B)= \rank(A'[W] - \lambda I) + 2 = |G| - 1+2=|G| +1,\]
a contradiction.  Therefore, $a'_{ij}- \delta_{ij}\lambda\neq 0$. As this argument applies to any $i,j\in V(K_s)$ we conclude $A'\in \calS(H')$.  
\end{proof}

\begin{center}
 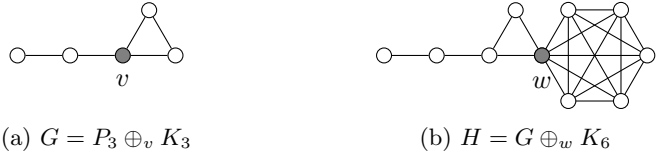
\begin{figure}[htb]
 \centering
 \begin{subfigure}{0.45\textwidth}
\centering
\begin{tikzpicture}[scale=0.7]
	\node (0) at (0,0) {};
  	\node[fill=gray, label={below:$v$}] (1) at (-1,0) {};
  	\node (2) at (-0.5,0.86) {};
  	\node (3) at (-2,0) {};
  	\node (4) at (-3,0) {};
             \draw  (0) -- (1) -- (2) -- (0);
             \draw  (4) -- (3) -- (1);
    \phantom{\begin{scope}[shift={(-1,0)}]
    \foreach \angle in {0, 60,120} {
              \draw  (\angle:1)--(\angle+180:1);
     }
    \foreach \angle in {0, 60,120,180,240,300} {
             \draw  (\angle:1)--(\angle+60:1);
              \draw  (\angle:1)--(\angle+120:1);}
 \foreach \angle in {0, 60,120,240,300} {
    \node[fill=white] at (\angle:1) {};
     } \end{scope}}
  \end{tikzpicture}
  \caption{$G=P_3 \oplus_v K_3$}
  \end{subfigure}
  \begin{subfigure}{0.45\textwidth}
\centering
\begin{tikzpicture}[scale=0.7]
	\node (0) at (0,0) {};
  	\node (1) at (-1,0) {};
  	\node (2) at (-0.5,0.86) {};
  	\node (3) at (-2,0) {};
  	\node (4) at (-3,0) {};
             \draw  (0) -- (1) -- (2) -- (0);
             \draw  (4) -- (3) -- (1);
  \begin{scope}[shift={(1,0)}]
    \foreach \angle in {0, 60,120} {
              \draw  (\angle:1)--(\angle+180:1);
     }
    \foreach \angle in {0, 60,120,180,240,300} {
             \draw  (\angle:1)--(\angle+60:1);
              \draw  (\angle:1)--(\angle+120:1);}
 \foreach \angle in {0, 60,120,240,300} {
    \node[fill=white] at (\angle:1) {};
     } \end{scope}
         \node[fill=gray, label={below:$w$}] at (0,0) {};
  \end{tikzpicture}
  \caption{$H=G \oplus_w K_6$}
  \end{subfigure}
     \caption{Two examples of graphs obtained by appending a clique.}\label{fig:clique appending}
     \end{figure}
  \end{center}

\begin{remark}\label{rem:submatrix}
In what follows we will want to apply Theorem \ref{thm:addeig} inductively. In this process we will keep track of the eigenvalues of $A$, but not of the eigenvalues of a submatrix $A(v)$. However, if we assume that $\mu \not\in \spec(A)\cup\spec(A(v))$ for any $v \in V(G)$ and $\mu \neq \lambda$, then we can guarantee that $\mu \not\in \spec(A')\cup\spec(A'(u))$ for any $u \in V(G \oplus_v K_{s+1})$, since $A'$ can be chosen to be arbitrarily close to $A \oplus \lambda I_s$.
\end{remark}

\begin{corollary}
\label{cor:appcliques}
Let $\lambda_1,\ldots,\lambda_h$ be distinct real numbers, $m_1,\ldots,m_h$  positive integers, and $G$ a graph. Let a sequence of graphs be defined as follows:
\begin{align*} 
H^{(0)}&=G, \\
H^{(k+1)}&=H^{(k)} \oplus_{v_{k+1}} K_{m_{k+1}+1} \text{ for some }v_{k+1} \in V(H^{(k)}). 
\end{align*}
Let $A\in\calS(G)$ be a matrix with the SSP such that $\lambda_i\notin\spec(A)\cup\,\spec(A(v))$ for all $v\in V(G)$ and $i \in \{1,\ldots,h\}$. Then there exists a matrix $\hat{A}\in\calS(H^{(h)})$ with the SSP and the spectrum 
\[\spec(\hat A)=\spec(A)\cup\{\lambda_1^{(m_1)},\ldots,\lambda_h^{(m_h)}\}.\]
\end{corollary}
\begin{proof}
We prove the result by induction on $h$. For $h=0$, we take $A=A^{(0)}$ and there is nothing to prove.

Assume now that there exists $A^{(i-1)} \in \calS(H^{(i-1)})$ with the SSP, $\spec (A^{(i-1)})=\spec(A) \cup \{\lambda_1^{(m_
1)},\cdots,\lambda_{i-1}^{(m_{i-1})}\}$, and none of $\lambda_i,\lambda_{i+1},\ldots,\lambda_h$ is an eigenvalue of $A^{(i-1)}(v)$ for any $v \in V(H^{(i-1)})$. By Theorem \ref{thm:addeig} applied to $A^{(i-1)}$, there exists 
$A^{(i)} \in \calS(H^{(i)})$ with the SSP and  
\[\spec(A^{(i)}) = \spec(A^{(0)})\cup\{\lambda_1^{(m_1)},\ldots,\lambda_i^{(m_i)}\}.\]
Finally, Remark \ref{rem:submatrix} allows us to assert $\{\lambda_{i+1},\ldots,\lambda_h\} \not\in A^{(i)}(v)$ for any $v \in V(H^{(i)})$. 
\end{proof}


\section{Block Graphs}
\label{sec:block}

While techniques developed in this work can be applied more broadly, we use a family of graphs known as block graphs as an illustrative example.

\begin{definition}
A \emph{block} of a graph is its maximal $2$-connected induced subgraph. 

A \emph{block graph} $G$ is a graph whose blocks are cliques.  The family of block graphs with blocks of sizes $m_1,m_2, \ldots, m_h$ will be denoted by $\mc{BG}(m_1,\ldots,m_h)$.  Consequently, every $G \in \mc{BG}(m_1,\ldots,m_h)$ has $1+\sum_{i=1}^h (m_i-1)$ vertices.
\end{definition}

 \begin{center}
 \begin{figure}[htb]
 \centering
\begin{subfigure}{0.3\textwidth}
\centering
\begin{tikzpicture}[scale=0.7]
	\node (0) at (0,0) {};
  	\node (1) at (90:1) {};
  	\node (2) at (150:1) {};
  	\node (3) at (210:1) {};
  	\node (4) at (270:1) {};
             \draw  (3) -- (4);
	    \foreach \x in {1,2,3,4} {
	                  \draw  (0) -- (\x);
	                    \node[fill=white] at (\x) {};
}
  \begin{scope}[shift={(1,0)}]
    \foreach \angle in {0, 60,120} {
              \draw  (\angle:1)--(\angle+180:1);
     }
    \foreach \angle in {0, 60,120,180,240,300} {
             \draw  (\angle:1)--(\angle+60:1);
              \draw  (\angle:1)--(\angle+120:1);}
      \foreach \angle in {0, 60,120,240,300} {
    \node[fill=white] at (\angle:1) {};
     }
  \end{scope}
       \node[fill=gray] at (0,0) {};
  \draw[opacity=0] (0,-1.3) -- ++(0,2.6); 
  \end{tikzpicture}
  \caption{$G_1$}
  \end{subfigure}
  \begin{subfigure}{0.3\textwidth}
\centering
\begin{tikzpicture}[scale=0.7]
	\node (0) at (0,0) {};
  	\node (1) at (-2,0.5) {};
  	\node[fill=gray] (2) at (150:1) {};
  	\node[fill=gray] (3) at (210:1) {};
  	\node (4) at (-2,-0.5) {};
             \draw  (0)--(2)--(3) -- (0);
              \draw  (2) -- (1);
              \draw  (4) -- (3);
  \begin{scope}[shift={(1,0)}]
    \foreach \angle in {0, 60,120} {
              \draw  (\angle:1)--(\angle+180:1);
     }
    \foreach \angle in {0, 60,120,180,240,300} {
             \draw  (\angle:1)--(\angle+60:1);
              \draw  (\angle:1)--(\angle+120:1);}
    \foreach \angle in {0, 60,120,240,300} {
    \node[fill=white] at (\angle:1) {};
     }
 \end{scope}
     \node[fill=gray] at (0,0) {};
  \draw[opacity=0] (0,-1.3) -- ++(0,2.6); 
  \end{tikzpicture}
  \caption{$G_2$}
  \end{subfigure}
\begin{subfigure}{0.33\textwidth}
\centering
\begin{tikzpicture}[scale=0.7]
  	\node[fill=gray] (1) at (-1,0) {};
  	\node (2) at (-2,0) {};
    \node (3) at (3,0) {};
  	\node (5) at (2.5,0.86) {};
             \draw  (0,0) -- (1) -- (2);
             \draw  (2,0) -- (3)--(5)--(2,0);
  \begin{scope}[shift={(1,0)}]
    \foreach \angle in {0, 60,120} {
              \draw  (\angle:1)--(\angle+180:1);
     }
    \foreach \angle in {0, 60,120,180,240,300} {
             \draw  (\angle:1)--(\angle+60:1);
              \draw  (\angle:1)--(\angle+120:1);}
 \foreach \angle in {60,120,240,300} {
    \node[fill=white] at (\angle:1) {};
     } \end{scope}
         \node[fill=gray] at (0,0) {};
         \node[fill=gray] at (2,0) {};
  \draw[opacity=0] (0,-1.3) -- ++(0,2.6); 
  \end{tikzpicture}
  \caption{$G_3$}
  \label{fig:G3}
  \end{subfigure}
     \caption{Three examples of block graphs $G_i \in \mc{BG}(6,3,2,2)$. The cut-vertices are colored gray.}\label{6322}
     \end{figure}
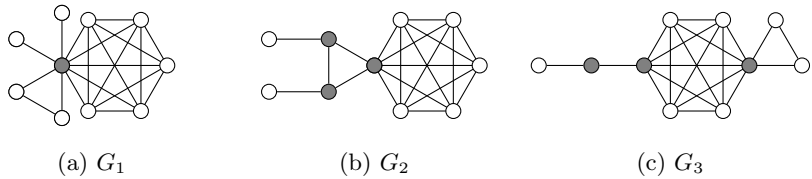
  \end{center}



\begin{definition}
A block graph $G$ is a \emph{minimal block graph} if every block in $G$ contains at most one non-cut vertex of $G$.  When $H$ is a block graph, the \emph{minimal block graph corresponding to $H$} is obtained from $H$ by removing (if any) all but one non-cut vertices in each block.
\end{definition}

Note that a block graph and its corresponding minimal block graph have the same cut vertices, and the same number of blocks. 
In Figure~\ref{6322} we present some examples of block graphs in $\mc{BG}(6,3,2,2)$, while their corresponding minimal block graphs are given in Figure~\ref{fig:minimal}. Note that every block graph is a blowup of its minimal block graph.

\begin{center}
\begin{figure}[htb]
\centering
\begin{subfigure}{0.3\linewidth}
  \centering
  \begin{tikzpicture}[scale=0.7]
      \node[fill=gray] (0) at (0:0) {};
      \foreach \i in {1,...,4} {
        \pgfmathsetmacro{\angle}{90*(\i-1)}
        \node (\i) at (\angle:1) {};
        \draw  (\i) -- (0);
       }
    \draw[draw=none] (0,-1.2) -- +(0,2.4); 
  \end{tikzpicture}
  \caption{$H_1\in\mc{BG}(2,2,2,2)$}
\end{subfigure}
\begin{subfigure}{0.3\textwidth}
\centering
\begin{tikzpicture}[scale=0.7]
	\node[fill=gray] (0) at (0,0) {};
  	\node (1) at (-2,0.5) {};
  	\node[fill=gray] (2) at (150:1) {};
  	\node[fill=gray] (3) at (210:1) {};
  	\node (4) at (-2,-0.5) {};
      \node (5) at (1,0) {};
             \draw  (0)--(2)--(3) -- (0)--(5);
              \draw  (2) -- (1);
              \draw  (4) -- (3);
    \draw[draw=none] (0,-1.2) -- +(0,2.4); 
   \end{tikzpicture}
  \caption{$H_2\in\mc{BG}(3,2,2,2)$}
\end{subfigure}
\begin{subfigure}{0.3\linewidth}
\centering
\begin{tikzpicture}[scale=0.7]
	\node[fill=gray] (0) at (0,0) {};
  	\node[fill=gray] (1) at (-1,0) {};
  	\node (2) at (-0.5,0.86) {};
  	\node[fill=gray] (3) at (-2,0) {};
  	\node (4) at (-3,0) {};
  	\node (5) at (1,0) {};
             \draw  (0) -- (1) -- (2) -- (0)--(5);
             \draw  (4) -- (3) -- (1);
    \draw[draw=none] (0,-1.2) -- +(0,2.4); 
   \end{tikzpicture}
  \caption{$H_3\in\mc{BG}(3,2,2,2)$}
\end{subfigure}
     \caption{Three examples of minimal block graphs, where for $i=1,2,3$, $H_i$ is the minimal block graph corresponding to $G_i$ in Figure~\ref{6322}. The cut-vertices are colored gray.}\label{fig:minimal}
     \end{figure}
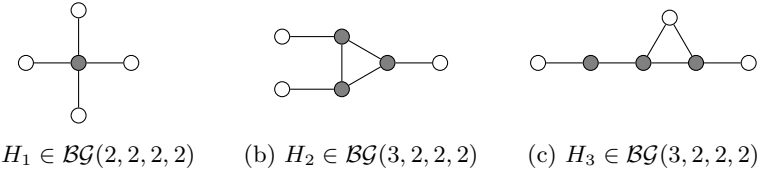
  \end{center}
 
 Block graphs  are also called block-clique graphs in \cite[Subsection~4.1]{MR3904092}. Following this terminology,  graph $G_1$ in Figure \ref{6322} is an example of a clique-star graph and $G_3$ in Figure \ref{6322} is an example of a clique-path graph. 

Complete graphs will be building blocks of our constructions. To avoid certain technical issues, we start this section by showing that realizations of spectra with two distinct eigenvalues in $\mc S(K_n)$ can be made generic enough.  The following lemma is a special case of \cite[Lemma~2.2]{MR364288}. We offer an alternative proof for completion.

\begin{lemma}\label{lemma:eig complete}
Let $\bx \in \R^n$ and $\by \in \R^m$. The matrix 
\begin{equation}\label{eq:blokrank1}
C=\npmatrix{\alpha_1 \bx\bx\trans+\beta_1 I_n & \gamma \bx\by\trans \\ \gamma \by \bx\trans & \alpha_2 \by\by\trans +\beta_2 I_{m}}
\end{equation}
has eigenvalues $(\mu_1,\mu_2,\beta_1^{(n-1)},\beta_2^{(m-1)}),$ where $\mu_1$ and $\mu_2$ are the eigenvalues of 
$$C'=\npmatrix{\alpha_1 \|\bx\|^2 +\beta_1 & \gamma \|\bx\|\, \|\by\| \\ \gamma \|\bx\| \, \|\by\| & \alpha_2 \|\by \|^2 +\beta_2}.$$
\end{lemma}

\begin{proof}
 Note that Sherman--Morrison formula implies that the inverse of matrix $Y_\lambda=\alpha_2\by\by\trans+(\beta_2-\lambda)I_m$ is equal to $Y_\lambda^{-1}=\frac{1}{\beta_2-\lambda}\left(I_m-\frac{\alpha_2\by\by\trans}{(\beta_2-\lambda)+\alpha_2\|\by\|^2}\right)$
 and that $\det(Y_\lambda)=(\beta_2-\lambda)^{m-1}(\beta_2-\lambda+\alpha_2\|\by\|^2)$.
 Using the Schur complement of a matrix $C-\lambda I_{m+n}$ we get  with some computation that the characteristic polynomial of matrix $C$ is equal to 
 \begin{align*}
 \det(C-\lambda I_{m+n})&=\det(Y_\lambda)\cdot \det(\alpha_1 \bx \bx\trans+(\beta_1-\lambda)I_n-\gamma^2\bx\by\trans Y_\lambda^{-1}\by\bx\trans) \\
 &=(\beta_1-\lambda)^{n-1}(\beta_2-\lambda)^{m-1}\det(C'-\lambda I_{2}).
\end{align*}
The assertion of the lemma follows.
\end{proof}

\begin{corollary}\label{cor:two distinct}
Let $\lambda_1, \lambda_2 \in \R$, $\lambda_1 \neq \lambda_2,$ $n_1$ and $n_2$ positive integers, $n=n_1+n_2$, and $\mc F$ and $\mc G$ finite sets of real numbers, so that $\{\lambda_1,\lambda_2\}\cap \mc F=\emptyset$.

Then there exists $A \in \calS(K_{n})$ with eigenvalues $\{\lambda_1^{(n_1)}, \lambda_2^{(n_2)}\}$ such that $A(v)$ has no eigenvalues in $\mc F$ for all $v \in V(K_{n})$, and $A$ has no diagonal elements in $\mc G$. 
\end{corollary}

\begin{proof}
Without loss of the generality, assume $\lambda_1<\lambda_2$. In order for the matrix $C$ of the form \eqref{eq:blokrank1} to have eigenvalues $\{\lambda_1^{(n_1)}, \lambda_2^{(n_2)}\}$, we choose in Lemma \ref{lemma:eig complete}  nowhere zero unit vectors $\bx \in \R^{n_1}$, $\by \in \R^{n_2}$, 
$\beta_1=\lambda_1$, $\beta_2=\lambda_2$, $\alpha_2=-\alpha_1$, $\gamma=\sqrt{\alpha_1 (\lambda_2-\lambda_1-\alpha_1)}$, and $\alpha_1$ so that $\alpha_1 (\lambda_2-\lambda_1-\alpha_1) >0.$  Note that $\alpha_1$ can be any real number in the open interval $(0,\lambda_2-\lambda_1)$, so it can be chosen so that the diagonal entries of $C$ avoid elements in $\mathcal{G}$; moreover, any small perturbation of $\alpha_1$ maintains the same property.

When $n_1=1$ and $v=1$, the spectrum of $C(v) =\alpha_2\by\by\trans + \beta_2 I_{n_2}$ is $\{-\alpha_1+\lambda_2, \lambda_2^{(n_2-1)}\}$ and can be chosen to avoid elements in $\mathcal{F}$.  The case for $n_2=1$ and $v=n$ follows from a similar argument.

Assume that $n_1>1$ and $n_2>1$. To complete the proof we note that $C(v)$ again has the form \eqref{eq:blokrank1} where either $\bx$ is replaced by $\bx(v)$ or $\by$ is replaced by $\by(v)$, and the eigenvalues of $C(v)$ can be deduced from  Lemma \ref{lemma:eig complete}. Thus, we may choose unit vectors $\bx$ and $\by$ so that the eigenvalues of matrices  
$$\npmatrix{\alpha_1 \|\bx(i)\|^2 +\beta_1 & \gamma \|\bx(i)\|\, \|\by\| \\ \gamma \|\bx(i)\| \, \|\by\| & \alpha_2 \|\by \|^2 +\beta_2} \text{ and } \npmatrix{\alpha_1 \|\bx\|^2 +\beta_1 & \gamma \|\bx\|\, \|\by(j)\| \\ \gamma \|\bx\| \, \|\by(j)\| & \alpha_2 \|\by(j) \|^2 +\beta_2},$$
$i=1,\ldots,n$, $j=1,\ldots,m$, avoid $\mc F$, while at the same time the diagonal elements of the matrix of the form \eqref{eq:blokrank1} avoid $\mc G$ for those vectors.
\end{proof}

In Theorem~\ref{cor:blockgraph} we will use matrices constructed in Corollary \ref{cor:two distinct} to produce a family of multiplicity lists that are realizable by a matrix with SSP for block graphs. 

\begin{definition}
The \emph{refinement} of a multiset of positive integers $\{a_1,\ldots,a_h\}$ is a multiset of positive integers 
\[\{b_{1,1},\ldots,b_{1,r_1}, b_{2,1},\ldots,b_{2,r_2},\ldots,b_{h,1},\ldots,b_{h,r_h}\}\]
such that $\sum_{j=1}^{r_i} b_{i,j} = a_i$ for every $i=1,\ldots,h$.

A multiset of positive integers $\{a_1,\ldots, a_\ell\}$ \emph{covers} another multiset of positive integers $\{b_1,\ldots, b_h\}$ if $\ell\geq h$ and there are $h$ elements $a_{i_1},\ldots, a_{i_h}$ such that $a_{i_j}\geq b_j$ for every $j=1,\ldots,h$.  Equivalently, the $i$-th largest element in $a_i$'s is greater than or equal to the $i$-the largest element in $b_i$'s for $i=1,\ldots,h$.
\end{definition}

\begin{figure}[h]
\begin{center}
\begin{tikzpicture}[scale=0.6]
  \foreach \i in {1,...,6} {
    \pgfmathsetmacro{\ang}{-90 + 60*\i}
    \node (\i) at (\ang:1) {};
  }
  \begin{scope}[yshift=-2.5cm]
  \foreach \i in {7,...,10} {
    \pgfmathsetmacro{\ang}{-135 + 90*(\i-6)}
    \node (\i) at (\ang:1) {};
  }
  \end{scope}
  \begin{scope}[yshift=-4.7cm]
  \foreach \i in {11,12,13} {
    \pgfmathsetmacro{\ang}{30 + 120*(\i-10)}
    \node (\i) at (\ang:1) {};
  }
  \end{scope}
  \foreach \i in {1,...,5} {
    \foreach \j in {\i,...,6} {
      \draw (\i) -- (\j);
    }
  }
  \foreach \i in {6,...,9} {
    \foreach \j in {\i,...,10} {
      \draw (\i) -- (\j);
    }
  }
  \foreach \i in {10,...,12} {
    \foreach \j in {\i,...,13} {
      \draw (\i) -- (\j);
    }
  }
  \node[rectangle, draw=none] at (0,-6.5) {$G$};
  \draw[draw=none] (0,-7) -- + (0.1,0); 
\end{tikzpicture}
\hfil
\begin{tikzpicture}
\node[draw=none] at (0,-4.7) {
\begin{tabular}{c|c||c|c}
 ~ & $G$ & $\hat{G}$ & ~ \\
 \hline
 $k$ & $2$ & $2$ & $k$ \\
 $m_1+1-k$ & $4$ & $2,2$ & $m_{1,1}, m_{1,2}$ \\
 $m_2$ & $4$ & $3,1$ & $m_{2,1},m_{2,2}$ \\
 $m_3$ & $3$ & $3$ & $m_{3,1}$ \\
\end{tabular}
};
\draw[draw=none] (0,-7) -- + (0.1,0); 
\end{tikzpicture}
\hfil
\begin{tikzpicture}[scale=0.6]
  \foreach \i in {1,...,6} {
    \pgfmathsetmacro{\ang}{-90 + 60*\i}
    \node (\i) at (\ang:1) {};
  }
  \begin{scope}[yshift=-2.5cm]
  \foreach \i in {7,...,10} {
    \pgfmathsetmacro{\ang}{-135 + 90*(\i-6)}
    \node (\i) at (\ang:1) {};
  }
  \end{scope}
  \begin{scope}[yshift=-4.7cm]
  \foreach \i in {11,12,13} {
    \pgfmathsetmacro{\ang}{30 + 120*(\i-10)}
    \node (\i) at (\ang:1) {};
  }
  \end{scope}
  \foreach \i in {1,...,3} {
    \foreach \j in {\i,...,4} {
      \draw (\i) -- (\j);
    }
  }
  \foreach \i in {4,...,5} {
    \foreach \j in {\i,...,6} {
      \draw (\i) -- (\j);
    }
  }
  \foreach \i in {6,...,8} {
    \foreach \j in {\i,...,9} {
      \draw (\i) -- (\j);
    }
  }
  \draw (9) -- (10);
  \foreach \i in {10,...,12} {
    \foreach \j in {\i,...,13} {
      \draw (\i) -- (\j);
    }
  }
  \node[rectangle, draw=none] at (0,-6.5) {$\hat{G}$};
  \draw[draw=none] (0,-7) -- + (0.1,0); 
\end{tikzpicture}
\end{center}
\caption{An illustration of $G\in \mc{BG}(6,5,4)$ and $\hat{G}$ in the proof of Theorem~\ref{cor:blockgraph}, following the chosen parameters from the table.}
\end{figure}
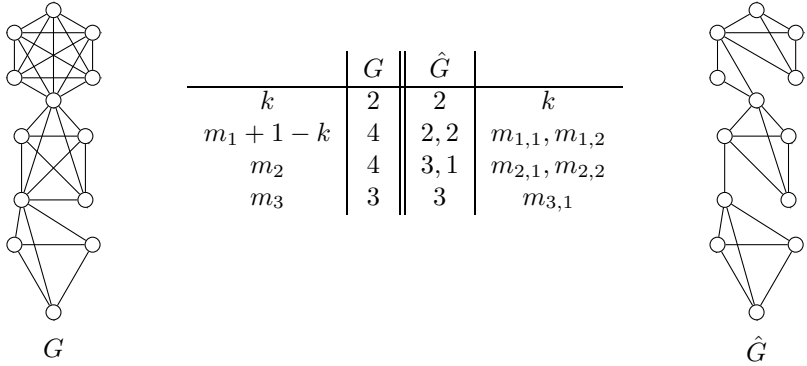

\begin{theorem}
\label{cor:blockgraph}
Let $m_1,\ldots,m_h$ be positive integers, $G \in \mc{BG}(m_1+1,m_2+1,\ldots,m_h+1)$, $1\leq k \leq m_1$, $\underline{r}$ a refinement  of $\{k,m_1+1-k,m_2,\ldots,m_h\}$, and $\sigma$ a multiset of real numbers with multiplicities $\underline{r}$. Then $\sigma$ can be realized by $A \in \calS(G)$ with the SSP.
\end{theorem}
\begin{proof}
Let
$$\underline{r}=\{k,m_{1,1},\ldots,m_{1,r_1},m_{2,1},\ldots,m_{2,r_2},\ldots, m_{h,1},\ldots,m_{h,r_h}\}$$
be a refinement of the multiset $\{k,m_1+1-k, \ldots, m_h\}$, where $\sum_{j=1}^{r_1} m_{1,j}=m_1+1-k$ and~$\sum_{j=1}^{r_i} m_{i,j}=m_i$ for $i=2,\ldots,h$. Note that we do not need to consider partitions of $k$, since they are covered by the freedom we have in choosing $k$ and the partition of $m_1+1-k$.

Choose a graph 
$G \in\mc{BG}(m_1+1,m_2+1, \ldots, m_h+1)$ and its subgraph 
$\hat G \in\mc{BG}(m_{1,1}+k,m_{1,2}+1,\ldots,m_{1,r_1}+1, \ldots, m_{h,1}+1, \ldots,m_{h,r_h}+1)$. Let
$\lambda_0,\lambda_1,\ldots,\lambda_r$ be any distinct real numbers, where $r=\sum _{i=1}^{h} r_i$. 

By Corollary \ref{cor:two distinct} we can choose a matrix $A^{(0)} \in \calS(K_{m_{1,1}+k})$ with spectrum $\{\lambda_0^{(k)},\lambda_1^{(m_{1,1})}\}$ such that $\lambda_2,\lambda_3,\ldots,\lambda_r \notin \spec(A^{(0)}) \cup \spec(A^{(0)}(v))$ for any $v \in V(K_{m_{1,1}+k})$.  
Corollary~\ref{cor:appcliques} allows us to construct a matrix 
$\hat{A}\in \calS(\hat G)$
 with SSP and the spectrum 
$\spec(\hat{A})=\{\lambda_0^{(k)},\lambda_1^{(m_{1,1})},\ldots,\lambda_r^{(m_{h,r_h})}\}$.
Since $\hat{A}$ has the SSP and $G$ is a supergraph of $\hat G$, it follows by Theorem~\ref{thm:sspsuper} that there exists a matrix $A \in \calS(G)$ with the SSP and spectrum $\spec(A)=\spec(\hat{A})=\{\lambda_0^{(k)},\lambda_1^{(m_{1,1})},\ldots,\lambda_r^{(m_{h,r_h})}\}$.
\end{proof}

\begin{remark}\label{rem:diags}
To obtain our next result, we will apply Lemma \ref{HS04} to the matrix $A$ constructed in the proof above. To do that we need some information on the diagonal elements of $A$. Following the construction of $A$ we can  deduce that the diagonal of $A$ can be made arbitrarily close to the diagonal of
$A^{(0)}\oplus \lambda_2 I_{m_{1,2}} \oplus \cdots \oplus \lambda_r I_{m_{h,r_h}}$.
 Moreover, by Corollary~\ref{cor:two distinct} we can choose $A^{(0)}$ so that its diagonal elements avoid any given finite set of real numbers $\mc G$, and by Theorem \ref{thm:addeig} we know the  rest of diagonal elements, while arbitrarily close to $\lambda_i$, are not equal to $\lambda_i$.
\end{remark}

When a block graph is not minimal, we can extend Theorem \ref{cor:blockgraph} using the approach from Section \ref{sec:dup}.

\begin{theorem}\label{thm:multilist minBG}
Let $G \in \mc{BG}(n_1+1,\ldots,n_h+1)$ be a block graph with the corresponding minimal block graph $G_0 \in \mc{BG}(m_1+1,\ldots,m_h+1)$.  Then any multiset of positive integers $\hat{m}_0,\ldots,\hat{m}_\ell$ with $\sum_{j=0}^\ell \hat{m}_j = |G|$ that covers a refinement of $\{k,m_1+1-k,m_2,\ldots,m_h\}$ for some $k$ with $1\leq k\leq m_1$ is a spectrally arbitrary multiplicity list for $G$. 
\end{theorem} 

\begin{proof}
Let $\sigma=\{\lambda_0^{(\hat m_0)}, \ldots, \lambda_\ell^{(\hat m_\ell)}\}$ be the spectrum that we want to achieve for a matrix $A \in \calS(G)$, and let $\underline{r}$ be a refinement of $\{k,m_1+1-k,m_2,\ldots,m_h\}$ that is covered by $\hat{m}_0, \ldots, \hat{m}_\ell$.  Thus, we may make a partition $\sigma = \sigma_0\cup\sigma'$ such that $\sigma_0$ has multiplicity list $\underline{r}$.

By Theorem~\ref{cor:blockgraph} and Remark~\ref{rem:diags}, there exists a matrix $A_0 \in \calS(G_0)$ whose spectrum is $\sigma_0$ and whose diagonal entries avoid elements in $\sigma'$. Since $G$ is a blowup of $G_0$, Lemma~\ref{lem:blowup} guarantees the existence of a matrix $A\in\calS(G)$ whose spectrum is $\spec(A_0)\cup\sigma' = \sigma$.
\end{proof}

Recall, the clique-path graphs are
$$KP(b_1,\ldots,b_h)=K_{b_1}\oplus_{v_1}K_{b_2}\oplus_{v_2}\cdots \oplus_{v_{h-1}}K_{b_h} \in \mc{BG}(b_1,\ldots,b_h),$$ 
where $b_i \geq 2$, $h\geq 1$, and $\{v_i\}_{i=1}^{h-1}$ are distinct. They are also known as the line graphs of caterpillars. The subproblems of IEP-$G$ of finding $q$ and $\mr$ are already solved for this family. Namely, $\mr(KP(b_1,\ldots,b_h))=h$ \cite[Table~1]{MR2388646}, and $q(KP(b_1,\ldots,b_h))=h+1$ \cite[Theorem~4.3]{MR3904092}. 

Recall that Corollaries~\ref{cor:lollipop} and \ref{cor:barbell} solve IEP-$G$ for special cases of clique-path graphs, where the only cliques of size greater than two are allowed to be at the ends of the clique-path graphs, namely for lollipop and generalized barbell graphs. However, Theorem~\ref{thm:multilist minBG} now resolves IEP-$G$ for clique-path graphs with an additional clique of size greater than two. Example of such graph is shown on~Figure~\ref{fig:G3}. 

\begin{corollary}
Let $G = KP(b_1,\ldots,b_h)$ be a clique-path graph such that $b_i=2$ for all $i$ except possibly for $i=1$, $i=h$ and at most one other index $j$, $2 \leq j \leq h-1$.
If $\sigma$ is a multiset with $1+\sum_{i=1}^h(b_i-1)$ elements, then $\sigma$ is a spectrum of a matrix in $\calS(G)$ if and only if $\sigma$ contains at least $h+1$ distinct elements.
\end{corollary}

\begin{proof}
Let $G$ and $j$ be as in the statement. If $b_j=2$, then $G$ is a generalized barbell graph, so this case is covered by Corollary~\ref{cor:barbell}.  Therefore, we may assume $b_j\geq 3$.

Recall that $q(G)=h+1$, hence every matrix in $\calS(G)$ has at least $h+1$ distinct element.
To prove the converse, observe that the corresponding minimal block graph of $G=KP(b_1,\ldots,b_h)$ is equal to $G_0=KP(c_1,\ldots,c_h)$, where $c_j=3$ and $c_i=2$ for all $i\in \{1,\ldots,h\}\setminus\{j\}$. Hence, $G_0 \in \mc{BG}(m_1+1,\ldots,m_h+1)$, where $m_j=2$ and $m_i=1$ for $i\in \{1,\ldots,h\}\setminus\{j\}$. By Theorem~\ref{thm:multilist minBG} any multiset of positive integers $\{\hat{m}_0,\ldots,\hat{m}_\ell\}$ with $\sum_{j=0}^\ell \hat{m}_j = |G| \geq h+1$ that covers multiset $\{2,1,\ldots,1\}$ with $h+1$ elements or $\{1,1,\ldots,1\}$ with $h+2$ elements, is a spectrally arbitrary multiplicity list for $G$. Hence, any multiset with at least $h+1$ distinct elements can be realized as a spectrum of a matrix in $\calS(G)$.
\end{proof}

Theorem \ref{thm:multilist minBG} exposes a large family of spectrally arbitrary multiplicities for block graphs, but does not solve the IEP-$G$ for block graphs in general. We illustrate this point with selected examples below.

\begin{example}
Theorem \ref{thm:multilist minBG} will always produce multiplicity lists with at least $h+1$ distinct elements for matrices in clique-path graphs $KP(b_1,\ldots,b_h)$, but in general will not solve the IEP-$G$.

For example, if $G:=KP(2,3,3,\ldots,3,2)$, it is a minimal block graph  with $h$ blocks, $|G|=2h-1$. Theorem \ref{cor:blockgraph} 
gives us spectrally arbitrary multiplicities 
$$\{\underbrace{2,\ldots,2}_{a},\underbrace{1,\ldots,1}_{b}\}, b \geq 3, a+b \geq h+1, 2a+b=2h-1.$$
Hence, we obtain all possible multiplicity lists with at least $h+1$ elements consisting  of only $1$'s and $2$'s. For $h=3$ this is a complete possible list of multiplicities for $KP(2,3,2)$, see \cite[Fig.~1]{MR4074182}. For $h \geq 4$, the maximal multiplicity of $KP(2,3,3,\ldots,3,2)$ with $h$ blocks is $h-1$ .
In particular, if $h=4$, then it remains to be resolved whether $\{3,1,1,1,1\}$ is spectrally arbitrary for the graph $KP(2,3,3,2)$.
\end{example}

\begin{example}
The minimal block graph for a clique-star graph $$G:=KS(m_1,\ldots,m_h) \in \mc{BG}(m_1,\ldots,m_h),$$ $m_i \geq 2$, $h\geq 2$, is equal to $K_{1,h}\in\mc{BG}(2,2,\ldots,2)$.

By Theorem \ref{thm:multilist minBG}, a multiplicity list $\{\hat{m}_0,\hat m_1,\ldots,\hat m_\ell\}$, $\ell \geq h$, is spectrally arbitrary for $G$ if it covers $\{1,1,\ldots,1\}$,  where $1$ is repeated $(h+1)$-times.  Equivalently, all partitions of $|KS(m_1,\ldots,m_h)|$ with at least $h+1$ parts are spectrally arbitrary multiplicity lists  for $KS(m_1,\ldots,m_h)$.

On the other hand, it is known \cite[Theorem~4.3]{MR3904092} that $q(KS(m_1,\ldots,m_h))=3$. It is also worth noting the multiplicity lists with $3$ elements may not be spectrally arbitrary for $KS(m_1,\ldots,m_h)$. For example, the only spectrally arbitrary multiplicity list for the minimal block graph  $K_{1,n}=KS(2,2,\ldots,2)\in \mc{BG}(2,2,\ldots,2)$ found by our method is $\{1,1,\ldots,1\}$. As it happens this is the only spectrally arbitrary multiplicity list in this case, \cite[Remark~4]{MR1902112}.
\end{example}

\begin{example}
The corona of a complete graph $K_n \circ K_1 \in \mc{BG}(n,2,2,\ldots,2)$ is a minimal block graph.
Theorem \ref{cor:blockgraph} tells us that multiplicity lists of the form
$$\{n_1,n_2,\ldots,n_t,\underbrace{1,\ldots,1}_{n}\}, \, t \geq 2, \, \sum_{j=1}^t n_j=n,$$
are spectrally arbitrary in this case. In particular, choosing $t=2$, $n_1=n-1$ and $n_2=1$, we obtain the spectrally arbitrary multiplicity list $\{n-1,1,\ldots,1\}$ for $K_n \circ K_1$, which achieves the maximal multiplicity $M(K_n \circ K_1)=n-1$ \cite[Table~1]{MR2388646}.

For example, for $G_{94}=K_3 \circ K_1$ the multiplicity list $\{2,1,1,1,1\}$ is spectrally arbitrary. This example and its multiplicity lists was considered in \cite[p.~35]{2017arxiv170802438}, where it was proved that ordered multiplicity lists $(1,2,1,2)$ and $(2,1,2,1)$ are realizable, while ordered multiplicity lists $(2,2,1,1)$, $(1,1,2,2)$ and $(1,2,2,1)$ are not, hence proving that our methods gives all possible spectrally arbitrary multiplicity lists for $G_{94}$. 
\end{example}

\subsection*{Acknowledgments}
Jephian C.-H. Lin was supported by the Young Scholar Fellowship Program (grant no.\ MOST-109-2636-M-110-006) from the Ministry of Science and Technology in Taiwan. Polona Oblak acknowledges the financial support from the Slovenian Research Agency  (research core funding No. P1-0222).
 

\end{document}